\documentclass[twoside,reqno]{amsart}
\usepackage{amsmath}

\usepackage{amsfonts}
\usepackage{amsthm}
\usepackage{amssymb}
\usepackage{graphicx,psfrag,epsfig}
\usepackage{color}
\usepackage{mathrsfs}
\usepackage{pdfsync}
\usepackage{pst-all}
\usepackage{cite}
\usepackage{times}

\newcommand{\bqn}{\begin{equation}}
\newcommand{\eqn}{\end{equation}}
\newcommand{\bqnn}{\begin{equation*}}
\newcommand{\eqnn}{\end{equation*}}
\newcommand{\bear}{\begin{eqnarray}}
\newcommand{\eear}{\end{eqnarray}}
\newcommand{\bean}{\begin{eqnarray*}}
\newcommand{\eean}{\end{eqnarray*}}

\newtheorem{theorem}{Theorem}[section]
\newtheorem{corollary}[theorem]{Corollary}
\newtheorem{lemma}[theorem]{Lemma}

\newtheorem{definition}[theorem]{Definition}

\numberwithin{equation}{section}
\begin{document}
\title{ Uniqueness of Weak Solutions to \\ a Prion Equation with Polymer Joining}

\author{Elena Leis}
\address{Leibniz Universit\"at Hannover, Institut f\" ur Angewandte Mathematik, Welfengarten 1, D--30167 Hannover, Germany} 
\email{leis@ifam.uni-hannover.de}

\author{Christoph Walker}
\address{Leibniz Universit\"at Hannover, Institut f\" ur Angewandte Mathematik, Welfengarten 1, D--30167 Hannover, Germany} 
\email{walker@ifam.uni-hannover.de}
%
%
\begin{abstract}
 We consider a model for prion proliferation that includes prion polymerization, polymer splitting, and polymer joining. The model consists of an ordinary differential equation for the prion monomers and a hyperbolic nonlinear dif\-fe\-ren\-ti\-al equation with integral terms for the prion polymers and was shown to possess global weak solutions for unbounded reaction rates \cite{LeisWalker}. Here we prove the uniqueness  of weak solutions. 
\end{abstract}
%
\maketitle
\pagestyle{myheadings}
\markboth{\sc{E. Leis \& Ch. Walker}}{\sc{A Prion Equation with Polymer Joining}}
%
%
\section{Introduction} \label{sec:int}

 In this article we consider a mathematical model for the dynamics of prions which are thought to be misfolded proteins and cause deadly neurodegenerative diseases  including ``mad cow disease''
in mammals. The model describes the proliferation of prions and was introduced in \cite{GvDWW07} to which we refer for more information regarding the biological background. The infectious prions are treated as polymers and interact with the noninfectious monomer form. 
The model includes polymerization, polymer joining,
and polymer splitting. These  processes can mathematically be described by a coupled system consisting of
an ordinary differential equation for the number of noninfectious 
monomers $v(t)\ge 0$, given by
\begin{equation}
\label{eqv}
\begin{split}
v'(t) & = \lambda-\gamma v(t) - \frac{v(t)}{1+\nu \displaystyle\int_{y_0}^{\infty}  u(t,z) z \mathrm{d}z} \int_{y_0}^{\infty} \tau(y)  u(t,y) \, \mathrm{d}y 
\\ & \quad + 2\int_{y_0}^{\infty} u(t,y) \beta (y) \int_0^{y_0} z \kappa(z,y)  \, \mathrm{d}z \, \mathrm{d}y \,,  
\end{split}
\end{equation}
and an integro-differential equation for the
density distribution function $u(t,y)\ge 0$ of infectious 
polymers of size $y>y_0$ of the form
 \begin{equation}\label{equ}
\begin{split}
 \partial_tu(t,y) + & \frac{v(t)}{1+\nu \displaystyle\int_{y_0}^{\infty}   u(t,z) z \mathrm{d}z}\partial_y{\left(\tau(y) u(t,y) \right)} \\ 
 & =  -(\mu(y)+\beta(y)) u(t,y)
 + 2 \int_y^{\infty} \beta(z) \kappa(y,z) u(t,z)  \, \mathrm{d}z \\ 
& \quad +\mathbf{1}_{[y>2y_0]}\int_{y_0}^{y-y_0}\eta (y-z,z) u(t,y-z) u(t,z) \, \mathrm{d}z
 - 2u(t,y)\int_{y_0}^{\infty}\eta (z,y) u(t,z) \, \mathrm{d}z
\end{split}
\end{equation}
for $y \in  Y:=(y_0, \infty)$. These equations are supplemented with the boundary condition
\begin{equation}
\label{RB0}
u(t,y_0) = 0 \, ,  \quad t>0  \,,
\end{equation}
and the initial values
\begin{equation}
\label{AW}
v(0) = v^0 \ , \quad u(0,y) = u^0(y) \, , \quad y \in (y_0, \infty) \, .
\end{equation}      
Here, $\lambda$  is a
constant monomer background source while $\gamma$ and $\mu(y)$ are the metabolic degradation rates for monomers, respectively, $y$-polymers. The function $\beta=\beta(y)$ is the splitting rate  for a polymer of size  $y$ into two polymers of size $z$ and $y-z$, where $\kappa(z,y)$ is the probability (density)
for this event. Any
daughter polymer with size less than the critical size $y_0>0$ is assumed to
disintegrate instantaneously into monomers. In the polymerization process,
infectious polymers of size $y>y_0$ attach noninfectious monomers at rate $\tau(y) > 0$. If $\nu>0$ there is a saturation effect when the number $\int_{y_0}^{\infty}  u(t,z) z \mathrm{d}z$ of monomers forming the infectious polymers becomes large resulting in less lengthening overall. Two polymers of size $y$ and $z$ may join at rate $\eta(y,z)$. Note that equation \eqref{equ} is reminiscent of the continuous coagulation-fragmentation equation known from physics (see e.g. \cite{FL} and the references therein).

For the case $\eta\equiv 0$, that is, when the bilinear polymer joining terms are neglected, equations \eqref{eqv}-\eqref{AW} were studied in \cite{EPW06,LW07,SW06,W} with respect to existence and uniqueness and in \cite{CLDDMP,CLODLMP,DG10,EPW06,G15,SW06} with respect to qualitative aspects.
The model with polymer joining was introduced in \cite{GvDWW07}. There it was assumed that  the rates have the particular form
 \begin{equation}\label{5}
     \tau \equiv \text{const}\,,\quad  \mu \equiv \text{const} \,,\quad  \eta \equiv \text{const}\,,\quad
\beta(y) =      \beta  y\,,\quad \kappa(z,y)  = \frac{1}{y}\,.
    \end{equation}
In this case, equation \eqref{equ} can be integrated and a closed system of ordinary differential equations for the unknowns $v$, $\int_{y_0}^{\infty}  u(y) \mathrm{d}y$, and $\int_{y_0}^{\infty}  u(y) y \mathrm{d}y$ can be derived that can be globally solved. The equations \eqref{eqv}-\eqref{equ} then decouple since $v$ is determined (see also \cite{EPW06,GPW06,PP-MWZ06}).

When polymer joining is taken into account, but \eqref{5} is not assumed, the existence of solutions to \eqref{eqv}-\eqref{AW} was established in \cite{LeisWalker}. More precisely, it was shown that for bounded reaction rates $\mu$, $\beta$, $\eta$, and $\tau$ a unique global classical solution exists. For unbounded (and thus biologically more relevant) reaction rates satisfying certain growth restrictions, the existence of global weak solutions was established. However, the uniqueness of weak solutions was left open and it is the purpose of this article to fill this gap. We shall show herein that under reasonable growth conditions on the reaction rates there is at most one weak solution. We thus extend the result of \cite{LW07} to include polymer joining using the same techniques. We shall point out here that we  also use ideas from \cite{FL} on the coagulation equation to handle the latter (see also \cite{Cepeda} where similar techniques are used to investigate uniqueness for the coagulation-fragmenation equations).
The uniqueness result from the present work complements the existence result of \cite{LeisWalker} to provide
 the well-posedness of \eqref{eqv}-\eqref{AW} in the framework of weak solutions.

 Before introducing the notation of a weak solution we remark that solutions are supposed to preserve the number of monomers. More precisely, let the splitting kernel $ \kappa \geq 0$ be a measurable function defined on $\mathcal{K}:=\{(z, y); y_0< y < \infty, 0 < z < y\}$ satisfying the symmetry condition
\begin{equation}
\kappa(z, y)=\kappa(y-z, y) \ , \quad (z, y) \in \mathcal{K} \ ,
\label{bin_split}
\end{equation}
and being normalized according to
\begin{equation}
2 \int_0^y z \kappa(z, y) \, \mathrm{d} z= y \ , \quad  \text{a.a. } y \in Y \ .
\label{mon_pres}
\end{equation}
Then splitting conserves the number of monomers and \eqref{bin_split}, \eqref{mon_pres} imply  
\begin{equation}\label{int_kappa}
\int_0^y \kappa (z, y) \, \mathrm{d} z= 1 \ , \quad  \text{a.a. } y \in Y \ .
\end{equation}
Note that if $\kappa$ is  e.g. of the form
\begin{equation}\label{k0_1}
\kappa (z, y) = \frac {1}{y} k_0 \left(  \frac{z}{y} \right) \ , \quad   y>y_0 \ , \quad  0<z<y \ ,
\end{equation} 
with a non-negative integrable function $k_0$ defined on $(0,1)$ satisfying
\begin{equation}\label{k0_2}
k_0( y) = k_0(1- y) \ , \quad   y \in (0,1) \ , \qquad \int_0^1 k_0(y) \, \mathrm{d} y = 1\,,
\end{equation}
then conditions \eqref{bin_split}, \eqref{mon_pres}  hold. In particular, for $k_0 \equiv 1$ one  obtains the rate
\begin{equation*}
\kappa(z,y) = \frac{1}{y} \, , \quad y > y_0 \, , \quad 0 < z < y \, ,
\end{equation*}
from \eqref{5} as considered in \cite{EPW06, GvDWW07, GPW06}.
We further assume that the polymer joining rate $\eta$ is non-negative and symmetric, that is,
\begin{equation}\label{etasym}
 0 \le \eta(y,z) = \eta(z,y) \,,\quad  y,z \in Y\,.
\end{equation}
Throughout this article we assume that
$\kappa$ satisfies conditions \eqref{bin_split}, \eqref{mon_pres} while  $\eta \in W_{\infty, \mathrm{loc}}^1(Y \times Y)$ satisfies \eqref{etasym}. We also assume that 
$$
\lambda, \gamma, \nu \geq 0\,
$$
and that $\tau$ is a positive measurable function on $Y$ growing at most linearly.
It is then straightforward to check that \eqref{bin_split}, \eqref{mon_pres}, and \eqref{etasym} imply that any solution $(v,u)$ to \eqref{eqv}-\eqref{AW} satisfies (formally) the monomer balance law
\begin{equation}\label{monomererhaltend}
\begin{split}
v(t) + &\int_{y_0}^{\infty}y u(t,y)\mathrm{d}y- v^0 - \int_{y_0}^{\infty}y u^0(y)\mathrm{d}y \\ 
&= \lambda t - \gamma \int_{0}^{t}v(s)\mathrm{d}s - \int_{0}^{t}\int_{y_0}^{\infty}y \mu(y) u(s,y)\mathrm{d}y\mathrm{d}s
\end{split}
\end{equation}
at time $t$.  That is, the overall number of monomers changes only due to natural production or metabolic degradation. To keep track of the  biologically important
quantities 
$$
\int_{y_0}^{\infty}  u(t,y)  \mathrm{d}y\qquad\text{and}\qquad \int_{y_0}^{\infty}  u(t,y) y \mathrm{d}y
$$
of all polymers respectively monomers forming those polymers, we shall thus consider solutions with $u(t,\cdot)$ belonging to the positive cone of $L_1(Y,y\mathrm{d}y)$,  denoted by $L_1^+(Y,y\mathrm{d}y)$.\\

\begin{definition}\label{D1}
Given $v^0 >0$ and $u^0 \in L_1^+(Y,y\mathrm{d}y)$ we  call a pair $(v,u)$ a  monomer preserving (global) weak solution to \eqref{eqv}-\eqref{AW} provided
\begin{itemize}
\item[\rm(i)] $v \in C^1(\mathbb{R}^+)$  is a non-negative solution to \eqref{eqv},
\item[\rm(ii)] $u \in L_{\infty,\mathrm{loc}} \big(\mathbb{R}^+, L_1^+(Y,y\mathrm{d}y) \big)$  is  a weak solution to \eqref{equ}, that is, it satisfies for all $t > 0$ 
\begin{align} \label{mu_beta_u}
&[(s,y) \to \big( \mu(y) + \beta(y)\big) u(s,y)] \in L_1 \big( (0,t) \times Y \big) 
\\ \label{eta_u}
&[(s,y,z)\mapsto \eta(y,z)u(s,y)u(s,z)]\in L_1((0,t)\times Y\times Y)
\end{align}
and 
\begin{align*}
\int_{y_0}^{\infty}\varphi(y) & u(t,y)\mathrm{d}y - \int_{y_0}^{\infty}\varphi(y)u^0(y)\mathrm{d}y
\\ = \, & \int_{0}^{t} \frac{v(s)}{1 + \nu \| u(s)\|_{L_1(Y, y \mathrm{d}y)}} \int_{y_0}^{\infty}\varphi'(y) \tau(y) u(s,y) \mathrm{d}y \mathrm{d}s \\
&-\int_{0}^{t} \int_{y_0}^{\infty}   \varphi(y) \mu(y)  u(s,y)  \, \mathrm{d} y \mathrm{d}s\\
&
 + \int_{0}^{t} \int_{y_0}^{\infty}  u(s,y)  \beta(y)  \left(   -\varphi(y)  +  2 \int_{y_0}^{y}  \varphi(z)  \kappa(z,y)   \, \mathrm{d} z  \right)   \, \mathrm{d} y \mathrm{d}s
 \\
& + \int_{0}^{t}\int_{y_0}^{\infty}  \int_{y_0}^{\infty}  (\varphi(y+z) - \varphi(y) - \varphi(z))  \eta(y,z)  u(s,y)u(s,z)  \, \mathrm{d}z  \, \mathrm{d}y \mathrm{d}s  
\end{align*}
for any test function $\varphi \in W^1_\infty(Y)$,
\item[\rm(iii)] the balance law \eqref{monomererhaltend} holds.
\end{itemize} 
\end{definition}

Note that the weak formulation in (ii) above is obtained by testing \eqref{equ} against the test function $\varphi$ and using for the operator
\bqn\label{Q}
Q[w](y):=\mathbf{1}_{[y>2y_0]}\int_{y_0}^{y-y_0}\eta (y-z,z) w(y-z) w(z) \, \mathrm{d}z
 - 2w(y)\int_{y_0}^{\infty}\eta (z,y) w(z) \, \mathrm{d}z\,,
\eqn
the identity     
\begin{equation}
\label{tilde}
\int_{y_0}^{\infty} \varphi(y)  Q[w](y)  \, \mathrm{d}y  = \int_{y_0}^{\infty}  \int_{y_0}^{\infty}  \big(\varphi(y+z) - \varphi(y) - \varphi(z)\big)  \eta(y,z)  w(y)w(z)  \, \mathrm{d}y  \, \mathrm{d}z\,,
\end{equation}
which follows from the symmetry of $\eta$.

As pointed out above the existence of a global weak solution in the sense of Definition~\ref{D1} was obtained in \cite{LeisWalker} under fairly general conditions on the reaction rates. We next state conditions under which such solutions are unique.

\subsection*{Main Results}

We shall first state simplified  and hence more illustrative versions of our actual results.  To this end we temporarily assume that the rates are of the particular form
\bqn\label{AK1}
\beta(y) = B y^b, \quad \mu(y) = M y^m, \quad \tau(y) = S y^{\theta}
\eqn
for $y > y_0$ with $B,M,S \geq 0$,  $0\le b,m \le 2$, and $ 0 \leq \theta \leq 1$. Moreover, let us also assume that 
\bqn\label{AK2}
\kappa \text{ is given by } \eqref{k0_1}, \eqref{k0_2}, \text{ where }  [y \mapsto y k_0(y)] \in L_{\infty}(0,1)\,.
\eqn 
We set
$$
\xi(x) := 2b \int_{x}^{1} k_0(z) \mathrm{d}z - b + 2x k_0(x)\ , \quad x \in (0,1) \,,
$$
and
$$
\alpha := \max\big\{\sup_{x \in (0,1)} \xi(x)\,,\, m\,,\, \theta\big\}  \, .
$$ 
Then we have the following uniqueness result for weak solutions with finite higher moments.

\begin{theorem}\label{ThE}
Suppose \eqref{AK1}, \eqref{AK2}  with $\alpha\in (0,2]$. Further suppose that there is a constant $K_0$ such that, if $\alpha \in (0,1]$, then
\begin{equation}\label{A}
 \frac{\eta(y,z)}{(y+z)^{\alpha}} +  \frac{(y^{\alpha} \wedge z^{\alpha}) |\partial_y \eta(y,z)|}{y^{\alpha-1} z^{\alpha}}\leq K_0\, ,   \quad (y,z) \in Y \times   Y \, ,
\end{equation}
while if $\alpha \in (1,2]$, then
\begin{equation}\label{B}
 \frac{\eta(y,z)}{y z^{\alpha-1} + y^{\alpha-1}z} +  \frac{(y \wedge z)(y^{\alpha-1} + z^{\alpha-1}) |\partial_y \eta(y,z)|}{y^{\alpha-1} z^{\alpha}}\leq K_0\, ,   \quad (y,z) \in Y \times   Y \, .
\end{equation}
Then there is $\sigma \geq 1$ (large enough and depending on $k_0$, $b$, $m$, $\theta$) such that \eqref{eqv}-\eqref{AW} has for each initial value $(v^0, u^0)$ with $v^0>0$ and $u^0 \in L_1^+(Y, y^{\sigma} \mathrm{d}y)$  at most one monomer-preserving weak solution  $(v,u)$  in the sense of Definition~\ref{D1} with $u \in L_{\infty,\mathrm{loc}} \big(\mathbb{R}^+, L_1(Y, y^{\sigma} \mathrm{d}y) \big)$. 
\end{theorem}

Theorem~\ref{ThE} is a special case of a more general result stated in Theorem~\ref{T1}. The latter does actually not require structural assumptions on the reaction rates as in \eqref{AK1} but rather suitable growth conditions. Note, however, that it does not provide uniqueness of weak solutions in the natural phase space $L_1(Y, y\mathrm{d}y)$. For rates $\mu$ and $\beta$ with at most linear growth we can though improve Theorem~\ref{ThE} to obtain a uniqueness result in $L_1(Y, y\mathrm{d}y)$ which in particular includes the rates from \eqref{5}. The following theorem is a special case of a more general result, see Theorem~\ref{T2}.

\begin{theorem}\label{ThE2}
Suppose \eqref{AK1} with  $b, m \leq 1$ and \eqref{AK2}. Further let \eqref{A} hold with 
$\alpha := m$ and suppose that
$\xi(x) \geq 0$ for $x \in (0,1)$.
Then, given any $(v^0, u^0) \in (0,\infty) \times L_1^+(Y, y\mathrm{d}y)$ there exists at most one monomer-preserving weak solution  $(v,u)$ to \eqref{eqv}-\eqref{AW}  in the sense of Definition~\ref{D1}.
\end{theorem}

 As pointed out before, the above theorems extend the results from \cite{LW07} for the case $\eta\equiv 0$, i.e. without polymer joining. To include polymer joining herein we use ideas from \cite{FL} on the coagulation equation. Assumptions \eqref{A}, \eqref{B} correspond to the assumptions therein.\\

Combining now the uniqueness statements above with the existence results of \cite[Theorem 2.3, Proposition 2.4]{LeisWalker} we obtain the well-posedness of \eqref{eqv}-\eqref{AW} within the framework of weak solutions.

\begin{corollary}\label{C1}
{\bf (a)} Let the assumptions of Theorem~\ref{ThE} hold for $\alpha \in (0,1]$. Then there is $\sigma \geq 1$ (large enough) such that \eqref{eqv}-\eqref{AW} has for each initial value $(v^0, u^0)$ with $v^0>0$ and $u^0 \in L_1^+(Y, y^{\sigma} \mathrm{d}y)$  a unique global monomer-preserving weak solution  $(v,u)$  in the sense of Definition~\ref{D1} such that $u \in L_{\infty,\mathrm{loc}} \big(\mathbb{R}^+, L_1(Y, y^{\sigma} \mathrm{d}y) \big)$.\\
{\bf (b)} Let the assumptions of Theorem~\ref{ThE2} hold. 
Then, given any $(v^0, u^0) \in (0,\infty) \times L_1^+(Y, y\mathrm{d}y)$, there exists a unique global monomer-preserving weak solution  $(v,u)$ to \eqref{eqv}-\eqref{AW}  in the sense of Definition~\ref{D1}.
\end{corollary}

The restriction to $\alpha \in (0,1]$ in Corollary~\ref{C1} (a) ensures the existence of a solution  $(v,u)$ with $u$ belonging to $L_{\infty,\mathrm{loc}}(\mathbb{R}^+, L_1(Y, y^{\sigma} \mathrm{d}y))$ for  $u^0 \in L_1^+(Y, y^{\sigma}\mathrm{d}y)$. However, the existence of a solution without this additional constraint can be shown also for $\alpha \in (1,2]$, see \cite{LeisWalker}.

\section{Sharper Statements of the Main Results}\label{EsL}

We now state more general results than in Theorem~\ref{ThE} and Theorem~\ref{ThE2} that do not rely on structural conditions on the reaction rates as in \eqref{AK1}-\eqref{AK2}, but rather on growth conditions.

Let us recall that we assume throughout that $\lambda, \gamma, \nu \geq 0$, that
$\kappa$ satisfies conditions \eqref{bin_split}, \eqref{mon_pres} while  $\eta \in W_{\infty, \mathrm{loc}}^1(Y \times Y)$ satisfies \eqref{etasym}.  Then, as in \cite{LW07}, we shall further assume that
\begin{equation}\label{18}
\lim_{y* \searrow y} \int_{y}^{y_*}  \kappa(z,y_*) \mathrm{d} z = 0 \ , \quad y > y_0\,,  
\end{equation}
and
\begin{equation}
\label{mu_beta_E}
\mu, \beta \in W_{\infty, \mathrm{loc}}^{1} (Y) \quad \text{with} \quad \mu, \beta \geq 0 \, . 
\end{equation}
Let there be a strictly positive function
\begin{equation}
\label{g}
g \in W^{1}_{\infty, \mathrm{loc}}(Y) \quad \text{with} \quad g'(y) \leq c_0(g(y)+1) \, , ~  y>y_0 \ ,
\end{equation}
for some constant $c_0 > 0$
such that
\begin{equation}
\label{20} 
\begin{cases} \tau \in W^1_{\infty, \mathrm{loc}}(Y) \ ,  \quad 0 < \tau(y) \leq c_0 y \ , \quad y \geq y_0 \ , \\ |\tau'(y)| \leq c_0 g(y) \quad \text{and} \quad (\tau g)'(y) \leq c_0 g(y) \  , \quad y>y_0 \,. \end{cases} 	
\end{equation}
Introducing $g$'s primitive
$$
G(y) := \int_{y_0}^{y} g(z) \mathrm{d} z \ ,  \quad y > y_0 \,, 
$$
we shall further assume that
\begin{align}
\label{21}
(\mu + \beta)(y) &\leq c_0((\mu + \beta)(y_*) + G(y_*) + y_*) \ , && y_* > y > y_0 \ ,\\
\label{22}
|\mu'(y)| + |\beta'(y)| &\leq c_0 g(y) \ , && y > y_0 \ ,\\
\label{23}
\left| \partial_y \left( \beta(y) \int_{0}^{y_0} z \kappa(z,y) \mathrm{d}z \right) \right| + |B_2(y_*,y)| &\leq c_0 g(y)\ , && y > y_* \geq y_0 \,, 
\end{align}
where
\begin{equation}\label{24}
B_2(y_*,y) := \partial_y \left( \beta(y) \int_{y_*}^{y} \kappa(z,y) \mathrm{d} z \right)  \,, \quad  y > y_* \geq y_0 \, .
\end{equation}
In addition, assume that
\begin{equation}
\label{25}
\int_{y_0}^{y} g(y_*) |2 B_2(y_*,y) - (\beta' + \mu')(y)| \mathrm{d} y_* \leq g(y)(c_0 + (\mu + \beta)(y))\,,\quad   y > y_0 \, .
\end{equation}
As for the polymer joining rate $\eta  \in W_{\infty, \mathrm{loc}}^1(Y\times Y)$ we suppose that there is a constant $K > 0$ such that 
\begin{equation}
\label{eta_W}
\frac{\eta(y,z)}{G(y) + y } 
+ \frac{|\partial_y \eta(y,z)|}{g(y)} \leq K  \big(G(z) + z \big)    \,, \quad (y,z) \in Y \times Y\,,
\end{equation}
along with
\begin{align}
\label{eta_g}
\eta(y,z)|g(y+z) - g(y)| &\leq  K g(y) \big(G(z) + z \big) \, , && (y,z) \in Y \times Y\,,\\
\label{eta_G}
\eta(y,z)\Big[ G(y+z) - G(y\vee z) + G(y \wedge z)  \Big] &\leq K  \big(G(y) + y \big) \big(G(z) + z \big)  \, , && (y,z) \in Y \times Y\,,\\
\label{p_eta_G} 
\left|\partial_y \eta(y,z)\right| \Big[G(y+z) - G(y\vee z) + G(y \wedge z)  \Big]  &\leq K g(y) \big( G(z) + z \big)\, ,&& (y,z) \in Y \times Y\,.
\end{align}
We then assume that there is a constant $g_0>0$ such that
\begin{equation}
\label{32a}
\text{if\  $\nu > 0$,\ then \  $g(y) \geq g_0 > 0$\ for \ $y \in Y$\,.}
\end{equation}
For the first uniqueness result we also require that
\begin{equation}
\label{31}
\int_{y_*}^{y}  |(\mu' + \beta')(z)| \mathrm{d} z \leq c_1(1+(\mu + \beta)(y)) \ , \quad y > y_* > y_0 \ ,
\end{equation}
and
\begin{equation}
\label{32}
\int_{y'}^{y}  |B_2(y',y_*)| \mathrm{d} y_* \leq c_1(1+(\mu + \beta)(y)) \ , \quad y > y' > y_0 \ ,
\end{equation}
for some constant $c_1>0$.
Then we have the following uniqueness result for solutions with sufficient integrability.

\begin{theorem}\label{T1}
Let \eqref{18}-\eqref{32} be satisfied. Then, for any initial value  $(v^0,u^0)$ with $v^0 >0$ and $u^0 \in L_1^+(Y, y\mathrm{d}y) \cap L_1\big( Y, (\mu + \beta)(y) G(y) \mathrm{d}y \big)$ there is at most one monomer-preserving weak solution   $(v,u)$  to \eqref{eqv}-\eqref{AW} in the sense of Definition~\ref{D1} such that
$$
u \in L_{\infty, \mathrm{loc}} \big( \mathbb{R}^+, L_1(Y, G(y)\mathrm{d}y) \big) \cap L_{1, \mathrm{loc}} \big( \mathbb{R}^+, L_1(Y, (\mu + \beta)(y)G(y)\mathrm{d}y) \big) \, .
$$ 
\end{theorem}

Although Theorem~\ref{T1} applies to a wide class of reaction rates it does not include the rates from \eqref{5} to yield uniqueness when $u$ belongs to the natural phase space $L_{\infty, \mathrm{loc}} \big( \mathbb{R}^+, L_1(Y, y\mathrm{d}y) \big)$ but rather for $u\in L_{\infty, \mathrm{loc}} \big( \mathbb{R}^+, L_1(Y, y^2\mathrm{d}y) \big)$.  To remedy this issue we shall consider \eqref{18}-\eqref{p_eta_G} in the particular case
$g \equiv 1$ (then \eqref{32a} trivially holds)
and further suppose that there are $C_1 > 0$ and $\delta > 0$ such that
\begin{equation}
\label{eta_mu}
\eta(y,z) \leq C_1 (\mu(y) + \mu(z)) \ , \quad (y,z) \in Y \times Y \ ,
\end{equation}
and
\begin{equation}
\label{26}
\int_{0}^{y} \frac{y_*}{y} \left( 1 - \frac{y_*}{y} \right) \kappa(y_*,y) \mathrm{d} y_* \geq \delta \ , \quad y > y_0 \,.
\end{equation}
Moreover, we suppose that we can decompose $\mu$ and $\beta$ in the form
\begin{equation}
\label{mu_beta_1}
\begin{cases} 
\mu = \mu_1 + \mu_2  \quad \text{ with } \mu_1, \mu'_1 \geq 0 \ , \quad \mu'_2 \in L_1(Y) \ , 
\\ \beta = \beta_1 + \beta_2  \quad \text{ with } \beta_1, \beta'_1 \geq 0 \ , \quad \beta'_2 \in L_1(Y) \, 
\end{cases}
\end{equation}
and that there is a constant $C_2 > 0$ such that
\begin{equation}
\label{RB}
\frac{1}{R} \int_{y_0}^{R} | B_2(y_*,y)| \mathrm{d} y_* \leq C_2 \beta'_1(y) \ , \quad y > R > y_0 \ .
\end{equation}
Then we obtain the following uniqueness result in the natural phase space $ L_1(Y, y \mathrm{d}y)$.

\begin{theorem}\label{T2}
Suppose that \eqref{18}-\eqref{p_eta_G} hold with $g \equiv 1$ and let \eqref{eta_mu}-\eqref{RB} be satisfied.
Then, for any initial value  $(v^0,u^0)$ with $v^0 >0$ and  $u^0 \in L_1^+(Y, y \mathrm{d}y)$ there is at most one monomer-preserving weak solution   $(v,u)$  to \eqref{eqv}-\eqref{AW} in the sense of Definition~\ref{D1}.  
\end{theorem}

 In the next section we derive suitable estimates on the primitive of the difference of two solutions. This first proves Theorem~\ref{T1} and Theorem~\ref{T2} which then entail Theorem~\ref{ThE} and Theorem~\ref{ThE2}. 

\section{Proofs}

\subsection*{A Priori Estimates}\label{rA}

Throughout this section we suppose that \eqref{18}-\eqref{32a} are satisfied. Let  $(v^0,u^0) \in (0,\infty) \times L_1^+(Y, y \mathrm{d} y)$ be given and consider two monomer-preserving solutions $(v,u)$  and $(\hat{v}, \hat{u})$ to \eqref{eqv}-\eqref{AW} in the sense of Definition~\ref{D1} such that
$$
u, \hat{u} \in L_{1,\mathrm{loc}}\big( \mathbb{R}^+, L_1(Y, G(y)\mathrm{d}y) \big) \, .
$$
Let us point out that then, in particular,
\begin{equation}\label{u_hat_u}
u,\hat{u} \in L_{1,\mathrm{loc}} \big( \mathbb{R}^+, L_1(Y,r(y)\mathrm{d}y) \big)\,,
\end{equation}
where $r(y):= (\mu + \beta)(y) + G(y) + y$ for $y \in Y$ and that \eqref{21} implies
\begin{equation}
\label{27}
\lim_{y \rightarrow \infty} (\mu + \beta)(y) \int_{y}^{\infty} |\phi(y_*)| \mathrm{d} y_* = 0 \quad \text{for} \quad \phi \in L_1(Y, r(y) \mathrm{d} y) \,.
\end{equation}
We now define
$$
E(t,y) := \int_{y}^{\infty} (u-\hat{u})(t,y_*) \mathrm{d} y_* \ , \quad (t,y) \in \mathbb{R}^+ \times Y \,   . 
$$
Clearly, to prove that $u$ and $\hat u$ coincide it suffices show that $E\equiv 0$. Let us fix $T > 0$ in the following.
Note that 
\begin{equation}
\label{E}
E \in L_{1,\mathrm{loc}} \big(  \mathbb{R}^+ , L_1(Y, g(y) \mathrm{d}y) \big) 
\end{equation} 
and $\partial_y E = \hat{u} - u$. Using integration by parts it then follows as in \cite[Lemma 2.2]{LW07} that $E$ satisfies the evolution equation 
\begin{equation}\label{dtE}
\begin{split}
& \partial_t E(t,y)  + \frac{v(t)}{1 + \nu \int_{y_0}^{\infty} z u(t,z) \mathrm{d}z} \, \tau(y) \partial_y E(t,y)
\\ & =
\left( \frac{v(t)}{1 + \nu \int_{y_0}^{\infty} z u(t,z) \mathrm{d}z}-\frac{\hat{v}(t)}{1 + \nu \int_{y_0}^{\infty} z \hat{u}(t,z) \mathrm{d}z} \right) \tau(y) \hat{u}(t,y) - (\mu + \beta)(y) E(t,y)
\\ & \qquad - \int_{y}^{\infty} (\mu' + \beta')(y_*) E(t,y_*) \mathrm{d} y_*
\\ & \qquad + 2\int_{y}^{\infty} E(t,y_*) B_2(y,y_*) \mathrm{d}y_*
\\ & \qquad + \int_{y_0}^{y} \int_{y_0}^{y}
\mathbf{1}_{[y,\infty)}(z+y_*) \eta(y_*,z)
[u(t,y_*)u(t,z) - \hat{u}(t,y_*) \hat{u}(t,z)] \mathrm{d} y_* \mathrm{d} z
\\ & \qquad - \int_{y}^{\infty} \int_{y}^{\infty}
 \eta(y_*,z)[u(t,y_*)u(t,z) - \hat{u}(t,y_*) \hat{u}(t,z)] \mathrm{d} y_* \mathrm{d} z 
\end{split} 
\end{equation}
for $t\ge 0$ and $y\in Y$. Therefore, introducing for $0 \le t \le T$ and $y \in Y$
\begin{equation*}\label{H}
\begin{split}
H[u,\hat{u}]  (t,y)  := &  \frac{-v(t)}{1 + \nu \int_{y_0}^{\infty} z u(t,z) \mathrm{d}z}\tau(y) \partial_y E(t,y)
\\ & \quad + \left( \frac{v(t)}{1 + \nu \int_{y_0}^{\infty} z u(t,z) \mathrm{d}z}-\frac{\hat{v}(t)}{1 + \nu \int_{y_0}^{\infty} z \hat{u}(t,z) \mathrm{d}z} \right)  \tau(y) \hat{u}(t,y)
\\ & \quad - (\mu + \beta)(y) E(t,y)
- \int_{y}^{\infty} (\mu' + \beta')(y_*) E(t,y_*) \mathrm{d} y_*
\\ & \quad  + 2\int_{y}^{\infty} E(t,y_*) B_2(y,y_*) \mathrm{d}y_*
\end{split}
\end{equation*}
and
\begin{equation*}
\label{F}\begin{split}
F[u,\hat{u}](t,y) &:= \int_{y_0}^{y} \int_{y_0}^{y}
\mathbf{1}_{[y,\infty)}(z+y_*) \eta(y_*,z)
 [u(t,y_*)u(t,z) - \hat{u}(t,y_*) \hat{u}(t,z)] \mathrm{d} y_* \mathrm{d} z
\\ & \qquad - \int_{y}^{\infty} \int_{y}^{\infty}
 \eta(y_*,z)[u(t,y_*)u(t,z) - \hat{u}(t,y_*) \hat{u}(t,z)] \mathrm{d} y_* \mathrm{d} z\,,
\end{split}
\end{equation*}
it follows from $E(0,y) = 0$, $y \in Y$, that
\begin{equation}
\label{EHF}
E(t,y) = \int_{0}^{t} H[u,\hat{u}](s,y) \mathrm{d} s +  \int_{0}^{t} F[u,\hat{u}](s,y) \mathrm{d} s \ , \quad y\in Y \ , \quad 0 \le t \le T \,.
\end{equation}
Our aim is to  estimate
\begin{equation}\label{g_E}
\int_{y_0}^{\infty} g(y)~ |E(t,y)|~ \mathrm{d}y=\int_{y_0}^{\infty} g(y) ~ \mathrm{sign}(E(t,y)) ~E(t,y) \mathrm{d}y 
\end{equation}
 and then apply Gronwall's inequality to obtain   that $E\equiv 0$.  This is implied by the subsequent lemmata.

\begin{lemma}\label{L2}
There is $c(T) > 0$ such that
$$
|E(t,y_0)| \leq c(T) \int_{0}^{t} \int_{y_0}^{\infty} g(y) |E(s,y)| \mathrm{d} y \mathrm{d} s \ , \quad 0 \le t \le T \ .
$$
\end{lemma}

\begin{proof}
Choose $\varphi \equiv 1$ in Definition~\ref{D1} and recall from \eqref{Q} the definition of $Q$. Then integration by parts yields
\begin{equation*}
\begin{split}
|E(t,y_0)| &= \left| \int_{y_0}^{\infty} (u - \hat{u})(t,y_*) \mathrm{d} y \right|
\\ & \leq \left| \int_{0}^{t} \int_{y_0}^{\infty} \left( 2 \beta(y) \int_{y_0}^{y} \kappa(y_*,y) \mathrm{d}y_* - (\beta + \mu)(y) \right)  (u - \hat{u})(s,y) \mathrm{d} y \mathrm{d} s  \right| 
\\ &  \quad  + \left| \int_{0}^{t} \int_{y_0}^{\infty} \big( Q[u(s,\cdot)](y)  -  Q[\hat{u}(s,\cdot)](y) \big) \mathrm{d} y \mathrm{d} s  \right| 
\\ & \leq  \left| \int_{0}^{t} \left[  \left( - 2 \beta(y) \int_{y_0}^{y} \kappa(y_*,y) \mathrm{d}y_* + (\beta + \mu)(y) \right) E(s,y) \right]_{y = y_0}^{y = \infty} \mathrm{d} s  \right| 
\\ & \quad +\left| \int_{0}^{t} \int_{y_0}^{\infty} \left( 2 B_2(y_0,y) - (\beta' + \mu')(y) \right) E(s,y) \mathrm{d} y \mathrm{d} s \right|
\\ &  \quad + \left| \int_{0}^{t} \int_{y_0}^{\infty} \big( Q[u(s,\cdot)](y)  -  Q[\hat{u}(s, \cdot)](y) \big) \mathrm{d} y \mathrm{d} s  \right| \,.
\end{split}
\end{equation*} 
Owing to \eqref{int_kappa} 
and \eqref{27} (applied to $\phi=u-\hat u$, see \eqref{u_hat_u}),  the boundary term at $y=\infty$ vanishes and we thus deduce with the help of \eqref{mu_beta_E}, \eqref{22}, and \eqref{23} that 
\begin{equation}
\label{E_y_0}
\begin{split}
|E(t,y_0)| & \leq c\int_{0}^{t} |E(s,y_0)| \mathrm{d} s + c \int_{0}^{t} \int_{y_0}^{\infty} g(y) |E(s,y)| \mathrm{d} y \mathrm{d} s
\\ &  \quad + \left| \int_{0}^{t} \int_{y_0}^{\infty} \big( Q[u(s,\cdot)](y)  -  Q[\hat{u}(s, \cdot)](y) \big) \mathrm{d} y \mathrm{d} s  \right| \, .
\end{split}
\end{equation}
To estimate the last term we use \eqref{tilde} and obtain
\begin{equation}\label{f_E}
\begin{split}
&\left|  \int_{0}^{t} \int_{y_0}^{\infty} \big( Q[u(s,\cdot)](y)  -  Q[\hat{u}(s,\cdot)](y) \big) \mathrm{d} y \mathrm{d} s  \right|  
\\ 
& \qquad   =  \left| - \int_{0}^{t} \int_{y_0}^{\infty} \int_{y_0}^{\infty}  \eta(y,z) [u(s,y) u(s,z) - \hat{u}(s,y) \hat{u}(s,z)] \mathrm{d} z \mathrm{d} y \mathrm{d} s  \right|
\\
 &  \qquad \le \left| \int_{0}^{t} \int_{y_0}^{\infty} \int_{y_0}^{\infty}  \eta(y,z) \partial_y E(s,y)  u(s,z) \mathrm{d} z \mathrm{d} y \mathrm{d} s  \right| 
\\ & \qquad \qquad +  \left| \int_{0}^{t} \int_{y_0}^{\infty} \int_{y_0}^{\infty}  \eta(y,z) \partial_z E(s,z) \hat{u}(s,y) \mathrm{d} z \mathrm{d} y \mathrm{d} s  \right| \ .
\end{split}
\end{equation}
Integration by parts gives
\begin{equation*}
\begin{split}
\int_{0}^{t}  \int_{y_0}^{\infty} &\int_{y_0}^{\infty}  \eta(y,z) \partial_y E(s,y)  u(s,z) \mathrm{d} z \mathrm{d} y \mathrm{d} s 
\\ & = \int_{0}^{t} \int_{y_0}^{\infty} \Big[ \eta(y,z) E(s,y) \Big]_{y=y_0}^{y=\infty}  u(s,z) \mathrm{d} z  \mathrm{d} s 
\\ & \quad  - \int_{0}^{t} \int_{y_0}^{\infty} \int_{y_0}^{\infty} \partial_y \eta(y,z)  E(s,y)  u(s,z) \mathrm{d} z \mathrm{d} y \mathrm{d} s \,, 
\end{split}
\end{equation*}
where the boundary term at $y=\infty$ vanishes since the monotonicity of  $G$ and \eqref{u_hat_u} imply 
$$
 \big(y + G(y) \big) |E(s,y)| \leq \int_{y}^{\infty} \big(y_* + G(y_*) \big) |u(y_*) - \hat{u}(y_*)| \mathrm{d} y_* \rightarrow 0 \quad \text{as} \quad y \to \infty \, .
$$ 
Hence, together with \eqref{eta_W} and \eqref{u_hat_u} we  deduce for $t\in [0,T]$
\begin{equation*}
\begin{split}
\Big| \int_{0}^{t}  \int_{y_0}^{\infty} \int_{y_0}^{\infty} & \eta(y,z) \partial_y E(s,y)  u(s,z) \mathrm{d} z \mathrm{d} y \mathrm{d} s \Big| \\
&\leq c(T)\int_{0}^{t} |E(s,y_0)| \mathrm{d} s + c(T) \int_{0}^{t} \int_{y_0}^{\infty} g(y) |E(s,y)| \mathrm{d} y \mathrm{d} s \,. 
\end{split}
\end{equation*}
From this along with \eqref{E_y_0}, \eqref{f_E}, and the symmetry of $\eta$ put us in a position to apply Gronwall's inequality to conclude.
\end{proof}

We next estimate the difference between $v$ and $\hat{v}$.

\begin{lemma}\label{L3}
There is $c(T) > 0$ such that
$$
|(v-\hat{v})(t)| \leq c(T) \int_{0}^{t} \int_{y_0}^{\infty} g(y) |E(s,y)| \mathrm{d} y \mathrm{d} s \ , \quad 0 \le t \le T \,.
$$
\end{lemma}

\begin{proof}  The proof is the same as in \cite[Lemma~3.4]{LW07} except that we have to treat also the case $\nu>0$.
It readily follows from \eqref{eqv},  \eqref{20}, and \eqref{u_hat_u} that
\begin{equation*}
\begin{split}
|(v-\hat{v})(t)| 
&\leq c(T) \int_{0}^{t} |v(s) - \hat{v}(s)| \mathrm{d}s
\\ & \quad + \nu  \|\hat{v}\|_{L_{\infty}(0,T)} \int_{0}^{t}   \left| \int_{y_0}^{\infty} z (\hat{u} - u)(s,z)  \mathrm{d} z \right| \int_{y_0}^{\infty} \tau(y) u(s,y) \mathrm{d} y  \mathrm{d} s
\\ & \quad + \|\hat{v}\|_{L_{\infty}(0,T)} \int_{0}^{t} \left| \int_{y_0}^{\infty} \tau(y) (u-\hat{u})(s,y) \mathrm{d} y \right| \mathrm{d} s
\\ & \quad + 2 \int_{0}^{t} \left| \int_{y_0}^{\infty} (u-\hat{u})(s,y) \beta(y) \int_{0}^{y_0} y_* \kappa(y_*,y) \mathrm{d} y_* \mathrm{d} y \right| \mathrm{d} s 
\end{split}
\end{equation*}
and hence
\begin{equation*}
\begin{split}
|(v-\hat{v})(t)| 
 & \leq c(T) \int_{0}^{t} |v(s) - \hat{v}(s)| \mathrm{d}s
\\ & \quad + \nu  \|\hat{v}\|_{L_{\infty}(0,T)} \int_{0}^{t}   \left|   \Big[ z E(s,z) \Big]_{z=y_0}^{z=\infty}
 - \int_{y_0}^{\infty}  E(s,z) \mathrm{d} z \right| \int_{y_0}^{\infty} \tau(y) u(s,y) \mathrm{d} y  \mathrm{d} s
\\ & \quad + \|\hat{v}\|_{L_{\infty}(0,T)} \int_{0}^{t} \left|  \Big[ -\tau(y) E(s,y) \Big]_{y=y_0}^{y=\infty} + \int_{y_0}^{\infty} \tau'(y) E(s,y) \mathrm{d} y \right|
 \mathrm{d} s
\\ & \quad + 2 \int_{0}^{t} \left| \Big[  - \beta(y) \int_{0}^{y_0} y_* \kappa(y_*,y) \mathrm{d} y_* E(s,y) \Big]_{y=y_0}^{y=\infty}  \right| \mathrm{d} s
\\  & \quad + 2 \int_{0}^{t} \left| \int_{y_0}^{\infty} \partial_y \left( \beta(y) \int_{0}^{y_0} z \kappa(z,y) \mathrm{d} z  \right) E(s,y) \mathrm{d} y \right| \mathrm{d} s \ .
\end{split}
\end{equation*} 
Note that the boundary terms at $\infty$ vanish owing to \eqref{int_kappa}, \eqref{20},  \eqref{u_hat_u}, and \eqref{27}. Thus, from \eqref{20} and \eqref{23} we deduce
\begin{equation*}
\begin{split}
|(v-\hat{v})(t)| & \leq c(T) \int_{0}^{t} |E(s,y_0)| \mathrm{d} s + c \int_{0}^{t} \int_{y_0}^{\infty}  g(y) |E(s,y)| \mathrm{d} y \mathrm{d} s
\\ & \qquad + c(T) \int_{0}^{t} |v(s) - \hat{v}(s)| \mathrm{d}s 
\end{split} 
\end{equation*}
 for $t\in [0,T]$, where we recall \eqref{32a} for the case that $\nu>0$. Gronwall's lemma together with Lemma~\ref{L2} yield the claim. 
\end{proof}

For the first term in \eqref{EHF} we obtain:

\begin{lemma}\label{L4}
If $R>y_0$ and $0 \le t \le T$, then
\begin{equation*}
\begin{split}
\int_{0}^{t} \int_{y_0}^{R} & g(y) ~ \mathrm{sign}(E(s,y)) ~H[u,\hat{u}](s,y)~ \mathrm{d} y \mathrm{d} s\qquad \qquad \qquad \qquad \qquad \qquad 
 \\ & \leq  
c \int_{0}^{t} \left| \frac{v(s)}{1 + \nu \int_{y_0}^{\infty} z u(s,z) \mathrm{d}z}- \frac{\hat{v}(s)}{1 + \nu \int_{y_0}^{\infty} z \hat{u}(s,z) \mathrm{d}z} \right|   \int_{y_0}^{R} (1+G(y)) \hat{u}(s,y) \mathrm{d} y \mathrm{d} s
\\ &  \quad + c(T) \int_{0}^{t} \int_{y_0}^{R} g(y) |E(s,y)|  \mathrm{d} y \mathrm{d} s + V(t,R)  \ ,
\end{split}
\end{equation*}
where
\begin{equation*}
\label{V(t,R)}
\begin{split}
V(t,R) := &\, G(R) \int_{0}^{t} \int_{R}^{\infty} \big|(\mu' + \beta')(y) \big|~ |E(s,y)| \mathrm{d}y \mathrm{d}s
 \\ & + \int_{0}^{t} \int_{R}^{\infty} |E(s,y)| \int_{y_0}^{R} g(y_*)|B_2(y_*,y)| \mathrm{d}y_* \mathrm{d}y \mathrm{d}s \, .
\end{split}
\end{equation*}
\end{lemma}

\begin{proof} 
This can be shown exactly as  estimate (37) in  \cite{LW07}  and using Lemma~\ref{L2}.
\end{proof}

For the second term in \eqref{EHF} we note:

\begin{lemma}\label{L5}
If $0 \le t \le T$, then
\begin{equation}
\begin{split}
& \int_{0}^{t} \int_{y_0}^{\infty}  g(y) ~ \mathrm{sign}(E(s,y)) ~F[u,\hat{u}](s,y)~ \mathrm{d} y \mathrm{d} s \le c(T) \int_{0}^{t} \int_{y_0}^{\infty} g(y) |E(s,y)| \mathrm{d}y \mathrm{d}s
\, .
\end{split}
\end{equation}
\end{lemma}

\begin{proof} We  adapt parts of the proof of \cite[Proposition 3.3]{FL}.
 Let $0\le t\le T$. We set
$$\tilde{G}(s,y) := \int_{y_0}^{y} g(z) \mathrm{sign}(E(s,z)) \mathrm{d} z \ , \quad y \in Y  \, , \quad s\in [0,T]\,, $$ 
and note that, for $z, y_*\in Y$ and $s \in [0,T]$,
\begin{equation}\label{R_G}
\begin{split}
\big| \tilde{G}(s,z+y_*) - \tilde{G}(s,z \vee y_*) - \tilde{G}(s, z \wedge y_*)  \big| 
 &  =
\big| \tilde{G}(s,z+y_*) - \tilde{G}(s,z ) - \tilde{G}(s, y_*)  \big| \\
& \leq  G(z+y_*) - G(z \vee y_*) + G(z \wedge y_*)\,.
\end{split}
\end{equation}
Then Fubini's theorem (along with \eqref{eta_G}, \eqref{u_hat_u}) and $\tilde{G}(s,y_0) = 0$ imply
\begin{equation*}
\begin{split}
&\int_{0}^{t} \int_{y_0}^{\infty} g(y) ~ \mathrm{sign}(E(s,y)) ~F[u,\hat{u}](s,y)~ \mathrm{d} y \mathrm{d} s 
\\ & = \int_{0}^{t}  \int_{y_0}^{\infty} \partial_y \tilde{G}(s,y) \int_{y_0}^{y} \int_{y_0}^{y}  \mathbf{1}_{[y,\infty)}(z+y_*) \eta(z,y_*)
 \big(u(s,z)u(s,y_*) - \hat{u}(s,z)\hat{u}(s,y_*) \big)\mathrm{d} z \mathrm{d} y_* \mathrm{d} y \mathrm{d} s 
\\ & \quad - \int_{0}^{t} \int_{y_0}^{\infty} \partial_y \tilde{G}(s,y) \int_{y}^{\infty}  \int_{y}^{\infty} \eta(z,y_*) \big(u(s,z)u(s,y_*) - \hat{u}(s,z)\hat{u}(s,y_*) \big)\mathrm{d} z \mathrm{d} y_*
 \mathrm{d} y \mathrm{d} s
\\ & = \int_{0}^{t}  \int_{y_0}^{\infty}  \int_{y_0}^{\infty} \eta(z,y_*) \Big[ \tilde{G}(s,z+y_*) - \tilde{G}(s,z \vee y_*) - \tilde{G}(s, z \wedge y_*) \Big] 
\\ & \qquad \qquad \qquad \qquad \qquad \qquad \qquad  \times \big(u(s,z)u(s,y_*) - \hat{u}(s,z)\hat{u}(s,y_*) \big) \mathrm{d} z \mathrm{d} y_*\mathrm{d} s \\
&= \int_{0}^{t}  \int_{y_0}^{\infty}  \int_{y_0}^{\infty} \eta(z,y_*) \Big[ \tilde{G}(s,z+y_*) - \tilde{G}(s,z \vee y_*) - \tilde{G}(s, z \wedge y_*) \Big] 
\\ & \qquad \qquad \qquad \qquad \qquad \qquad \qquad  \times \big( u(s,z) + \hat{u}(s,z) \big) \mathrm{d} z  \big( u(s,y_*) - \hat{u}(s,y_*) \big)   \mathrm{d} y_*\mathrm{d} s \,,
\end{split}
\end{equation*}
where the last inequality follows from the symmetry of $\eta$.
Thus, introducing
$$
I(s,y_*) := \int_{y_0}^{\infty} \eta(y_*,z)  \Big[ \tilde{G}(s,z+y_*) - \tilde{G}(s,z \vee y_*) - \tilde{G}(s, z \wedge y_*) \Big]  \big( u(s,z) + \hat{u}(s,z) \big) \mathrm{d} z 
$$
we obtain
\begin{equation}\label{F_I}
\begin{split}
\int_{0}^{t} \int_{y_0}^{\infty} & g(y) ~ \mathrm{sign}(E(s,y)) ~F[u,\hat{u}](s,y)~ \mathrm{d} y \mathrm{d} s
= 
\int_{0}^{t}  \int_{y_0}^{\infty}  I(s,y_*)  \big( u(s,y_*) - \hat{u}(s,y_*) \big)   \mathrm{d} y_*\mathrm{d} s \,.
\end{split}
\end{equation}
 Then \eqref{eta_G}, \eqref{u_hat_u}, and \eqref{R_G} entail
$$
|I(s,y_*)| \leq c(T)(G(y_*) + y_*) \ , \quad y_* \in Y \ , \quad 0 < s < T \, .
$$
For technical reasons we introduce for fixed $S > y_0$ the truncation
$$
I_S(s,y_*) := \int_{y_0}^{S} \eta(y_*,z)  \Big[ \tilde{G}(s,z+y_*) - \tilde{G}(s,z) - \tilde{G}(s,y_*)  \Big] \big( u(s,z) + \hat{u}(s,z) \big) \mathrm{d} z  
$$
of $I(s,y_*)$. As above, \eqref{eta_G}, \eqref{u_hat_u}, and \eqref{R_G} entail that
\begin{equation}\label{Is_G}
|I_S(s,y_*)| \leq c(T)(G(y_*) + y_*) \ , \quad y_* \in Y \ , \quad 0 \leq s \leq T \,.
\end{equation}
Moreover, recalling that $\eta \in W^1_{\infty, \mathrm{loc}}(Y \times Y)$ we have
\begin{equation}
\begin{split} \label{partial_Is}
\partial_{y_*} I_S(s,y_*) & = \int_{y_0}^{S} \partial_{y_*}\eta(y_*,z)   \Big[ \tilde{G}(s,z+y_*) - \tilde{G}(s,z) - \tilde{G}(s,y_*)  \Big] \big( u(s,z) + \hat{u}(s,z) \big) \mathrm{d} z
\\ & \quad  + \int_{y_0}^{S} \eta(y_*,z)  \Big( \partial_{y_*} \tilde{G}(s,z+y_*) - \partial_{y_*} \tilde{G}(s,y_*)  \Big) \big( u(s,z) + \hat{u}(s,z) \big) \mathrm{d} z \, .
\end{split}
\end{equation}
It then follows from \eqref{R_G} and \eqref{p_eta_G} that
\begin{equation*}
\begin{split}
&\left\vert \int_{y_0}^{S} \partial_{y_*}\eta(y_*,z)   \Big[ \tilde{G}(s,z+y_*) - \tilde{G}(s,z) - \tilde{G}(s,y_*)  \Big] \big( u(s,z) + \hat{u}(s,z) \big) \mathrm{d} z\right\vert\\
& \qquad\qquad\le K g(y_*)\int_{y_0}^S (G(z)+z) \big( u(s,z) + \hat{u}(s,z) \big) \mathrm{d} z
\end{split}
\end{equation*}
while \eqref{eta_g} entails that
\begin{equation*}
\begin{split}
&\left\vert \int_{y_0}^{S} \eta(y_*,z)  \Big( \partial_{y_*} \tilde{G}(s,z+y_*) - \partial_{y_*} \tilde{G}(s,y_*)  \Big) \big( u(s,z) + \hat{u}(s,z) \big) \mathrm{d} z \right\vert \\
& \qquad\qquad\le
\left\vert \int_{y_0}^{S} \eta(y_*,z)  \Big( \big\vert g(y_*+z)-g(y_*)\big\vert +2 g(y_*)  \Big) \big( u(s,z) + \hat{u}(s,z) \big) \mathrm{d} z \right\vert\\
& \qquad\qquad\le g(y_*)\int_{y_0}^S \Big( K(G(z)+z)+2\eta(y_*,z)  \Big) \big( u(s,z) + \hat{u}(s,z) \big) \mathrm{d} z\,.
\end{split}
\end{equation*}
Consequently, \eqref{Is_G}, \eqref{partial_Is},  and \eqref{u_hat_u} imply that $I_{S}(s,\cdot) \in  W^1_{\infty}([y_0, R))$ for each $R >y_0$ and $0\le s\le T$.
Hence, we can rewrite \eqref{F_I} in the form
\begin{equation}
\label{gE_Is}
\begin{split}
&\int_{0}^{t} \int_{y_0}^{\infty} g(y) ~ \mathrm{sign}(E(s,y)) ~F[u,\hat{u}](s,y)~ \mathrm{d} y \mathrm{d} s
\\ & = 
\int_{0}^{t} \Bigg( \int_{y_0}^{\infty}  (I-I_S)(s,y_*)  \big( u(s,y_*) - \hat{u}(s,y_*) \big)   \mathrm{d} y_* 
- \Big[ I_S(s,y_*) E(s,y_*) \Big]_{y_*=y_0}^{y_*=\infty} 
\\ & \qquad \qquad + \int_{y_0}^{\infty} \partial_{y_*} I_S(s,y_*) E(s,y_*) \mathrm{d} y_*\  \Bigg) \mathrm{d} s \,. 
\end{split}
\end{equation}
 It follows from  \eqref{Is_G} that
\begin{equation}
\label{IsE_0}
|I_S(s,y_0)E(s,y_0)| \leq c(T) |E(s,y_0)| \, , \quad 0 \leq s \leq T \,,
\end{equation}
and, since $G$ is non-negative and non-decreasing,
\begin{equation*}
\begin{split}
|I_S(s,y_*)E(s,y_*)| & \leq 
c(T) \int_{y_*}^{\infty} \big(G(z) + z  \big) (u+\hat{u})(s,z) \mathrm{d}z \,,
\end{split}
\end{equation*} 
hence, by \eqref{u_hat_u},
$$
\lim_{y_* \to \infty} I_S(s,y_*)E(s,y_*) = 0 \, , \quad 0 \leq s \leq T \ .
$$
Therefore, \eqref{gE_Is} implies for each $S > y_0$ the equality
\begin{equation}
\label{gEF_Is}
\begin{split}
\int_{0}^{t} & \int_{y_0}^{\infty} g(y) \mathrm{sign}(E(s,y)) F[u,\hat{u}](s,y)  \mathrm{d}y \mathrm{d}s
\\  & =  \int_{0}^{t} \Bigg( \int_{y_0}^{\infty}  (I-I_S)(s,y_*)  \big( u(s,y_*) - \hat{u}(s,y_*) \big)   \mathrm{d} y_* 
 \\ & \qquad \qquad + \int_{y_0}^{\infty} \partial_{y_*} I_S(s,y_*) E(s,y_*) \mathrm{d} y_* 
 +  I_S(s,y_0) E(s,y_0)  \Bigg) \mathrm{d} s \ .
\end{split}
\end{equation}
 From \eqref{R_G} and \eqref{eta_G} we obtain
\begin{equation*}
\begin{split}
& \left| \int_{y_0}^{\infty} (I-I_S)(s,y_*)  \big( u(s,y_*) - \hat{u}(s,y_*) \big)   \mathrm{d} y_* \right|
\\ & \qquad \leq  \int_{y_0}^{\infty} | (I-I_S) (s,y_*) | \big( u(s,y_*) + \hat{u}(s,y_*) \big)   \mathrm{d} y_* 
\\ & \qquad \leq \int_{y_0}^{\infty} \int_{S}^{\infty} \eta(y_*,z)  \Big[ G(z+y_*) - G(z \vee y_*) + G(z \wedge y_*)  \Big] 
\\ & \qquad \qquad \qquad \qquad \qquad \qquad \times \big( u(s,z) + \hat{u}(s,z) \big) \, \mathrm{d} z \, \big( u(s,y_*) + \hat{u}(s,y_*) \big) \,  \mathrm{d} y_* 
\\ & \qquad \leq C \int_{y_0}^{\infty} \int_{S}^{\infty} \big(  G(y_*) + y_*\big) \big(G(z) + z \big)   
\\ & \qquad \qquad \qquad \qquad \qquad \qquad \times \big( u(s,z) + \hat{u}(s,z) \big) \, \mathrm{d} z \, \big( u(s,y_*) + \hat{u}(s,y_*) \big) \,  \mathrm{d} y_* 
\\ & \qquad \leq c(T) \int_{S}^{\infty} \big(G(z) + z  \big)  \big( u(s,z) + \hat{u}(s,z) \big)  \mathrm{d} z \,,
\end{split}
\end{equation*}
so that \eqref{u_hat_u} guarantees
\begin{equation}
\label{I-Is}
\lim_{S \to \infty}  \int_{y_0}^{\infty}  (I-I_S)(s,y_*)  \big( u(s,y_*) - \hat{u}(s,y_*) \big)   \mathrm{d} y_* = 0 \ , \quad 0 \leq s \leq T \,.
\end{equation}
Next, we invoke \eqref{p_eta_G}, \eqref{u_hat_u}, \eqref{E}, \eqref{R_G} and apply Lebesgue's theorem to get
\begin{equation}
\label{p_eta_R}
\begin{split}
 \lim_{S \to \infty}& \int_{y_0}^{\infty} \int_{y_0}^{S}  \partial_{y_*}\eta(y_*,z)   \Big[ \tilde{G}(s,z+y_*) - \tilde{G}(s,z) - \tilde{G}(s,y_*)  \Big]  \big( u(s,z) + \hat{u}(s,z) \big) \mathrm{d} z   E(s,y_*) \mathrm{d} y_*
\\ & =  \int_{y_0}^{\infty} \int_{y_0}^{\infty} \partial_{y_*}\eta(y_*,z)   \Big[ \tilde{G}(s,z+y_*) - \tilde{G}(s,z) - \tilde{G}(s,y_*)  \Big] 
 \big( u(s,z) + \hat{u}(s,z) \big) \mathrm{d} z E(s,y_*) \mathrm{d} y_* \  .
\end{split}
\end{equation}
Moreover, we have   
\begin{equation*}
\begin{split}
& \limsup_{S\to \infty}  \int_{y_0}^{\infty} \int_{y_0}^{S} \eta(y_*, z) \Big(  \partial_{y_*}  \tilde{G}(s, z+y_*) -  \partial_{y_*}  \tilde{G}(s, y_*)  \Big) 
\big(u(s,z) + \hat{u}(s,z) \big) \mathrm{d}z  E(s,y_*) \mathrm{d}y_* 
\\ & = \limsup_{S\to \infty}  \int_{y_0}^{\infty} \int_{y_0}^{S} \eta(y_*, z)  \Big( g( z+y_*) \mathrm{sign} \big( E(s,z+y_*) E(s,y_*) \big) - g( y_*)  \Big)
\\ & \qquad  \qquad  \qquad  \qquad  \qquad  \qquad  \qquad  \qquad  \qquad    \times  
\big(u(s,z) + \hat{u}(s,z) \big) \mathrm{d}z |  E(s,y_*) | \mathrm{d}y_* \\
& \le \limsup_{S\to \infty}  \int_{y_0}^{\infty} \int_{y_0}^{S} \eta(y_*, z)  \Big( g( z+y_*)  - g( y_*)  \Big) 
 \big(u(s,z) + \hat{u}(s,z) \big) \mathrm{d}z | E(s,y_*)| \mathrm{d}y_*
\end{split}
\end{equation*}
and thus, by \eqref{eta_g},
\begin{equation}
\label{eta_p_R}
\begin{split}
 \limsup_{S\to \infty} & \int_{y_0}^{\infty} \int_{y_0}^{S} \eta(y_*, z) \Big(  \partial_{y_*}  \tilde{G}(s, z+y_*) -  \partial_{y_*}  \tilde{G}(s, y_*)  \Big) 
\big(u(s,z) + \hat{u}(s,z) \big) \mathrm{d}z  E(s,y_*) \mathrm{d}y_*  
\\ & \le   \int_{y_0}^{\infty} \int_{y_0}^{\infty} \eta(y_*, z)  \Big( g( z+y_*)  - g( y_*)  \Big)   \big(u(s,z) + \hat{u}(s,z) \big) \mathrm{d}z | E(s,y_*)| \mathrm{d}y_* \,.
\end{split}
\end{equation}
We then pass to the limit $S\to\infty$ in \eqref{gEF_Is} and deduce from \eqref{R_G}, \eqref{partial_Is}  and \eqref{IsE_0}-\eqref{eta_p_R} that
\begin{equation*}
\begin{split}
&\int_{0}^{t} \int_{y_0}^{\infty} g(y) ~ \mathrm{sign}(E(s,y)) ~F[u,\hat{u}](s,y)~ \mathrm{d} y \mathrm{d} s 
\\ & \leq    \int_{0}^{t}  \int_{y_0}^{\infty} \int_{y_0}^{\infty}\eta(y,z) \big[g(y + z) - g(y) \big] ~ \big(u(s,z) + \hat{u}(s,z) \big) ~ \mathrm{d}z ~ |E(s,y)|  \mathrm{d}y \mathrm{d}s
\\ &  \quad +   \int_{0}^{t}  \int_{y_0}^{\infty} \int_{y_0}^{\infty} |\partial_y \eta(y,z)| \big[ G(s,y+z) - G(s,y \vee z) + G(s,y \wedge z)  \big] 
\\ &  \qquad \qquad \qquad \qquad \qquad \qquad \qquad \qquad \qquad \times (u(s,z) + \hat{u}(s,z)) \mathrm{d}z ~| E(s,y)|  \mathrm{d}y \mathrm{d}s \ 
\\ & \quad + c(T) \int_{0}^{t} |E(s,y_0)| \mathrm{d}s 
\, .
\end{split}
\end{equation*}
The inequalities \eqref{eta_g} and \eqref{p_eta_G} yield together with \eqref{u_hat_u}  and Lemma~\ref{L2} the assertion.
\end{proof}

We now obtain from Fatou's lemma and  \eqref{EHF}, \eqref{g_E} that  
\begin{equation*}
\begin{split}
\int_{y_0}^{\infty} g(y)|E(t,y)| \mathrm{d}y
 &\le \limsup_{R \to \infty} \int_{0}^{t} \int_{y_0}^{R} g(y) \mathrm{sign}(E(t,y)) H[u,\hat{u}](s,y) \mathrm{d} y \mathrm{d} s
\\ & \qquad  + \int_{0}^{t} \int_{y_0}^{\infty} g(y) \mathrm{sign}(E(t,y)) F[u,\hat{u}](s,y) \mathrm{d} y \mathrm{d} s 
\end{split}
\end{equation*}
for $0 \le t \le T$.  Lemma~\ref{L4} and Lemma~\ref{L5} then yield 
\begin{equation}
\label{g_sup_R}
\begin{split}
 \int_{y_0}^{\infty} & g(y)|E(t,y)| \mathrm{d}y\\
 & \leq    c \int_{0}^{t} \left| \frac{v(s)}{1 + \nu \int_{y_0}^{\infty} z u(s,z) \mathrm{d}z}- \frac{\hat{v}(s)}{1 + \nu \int_{y_0}^{\infty} z \hat{u}(s,z) \mathrm{d}z} \right|  \int_{y_0}^{\infty} (1+G(y)) \hat{u}(s,y) \mathrm{d} y \mathrm{d} s
\\ & \quad  +c(T) \int_{0}^{t} \int_{y_0}^{\infty} g(y) |E(s,y)|  \mathrm{d} y \mathrm{d} s  +\limsup_{R \to \infty}   V(t,R) \,.
\end{split}
\end{equation}
Using
\begin{equation*}
\begin{split}
\left| \frac{v(s)}{1 + \nu \int_{y_0}^{\infty} z u(s,z) \mathrm{d}z}- \frac{\hat{v}(s)}{1 + \nu \int_{y_0}^{\infty} z \hat{u}(s,z) \mathrm{d}z} \right|
& \leq |(v - \hat{v})(s)| 
\\ & \quad + \|\hat{v}\|_{\infty} \, \nu \left| \int_{y_0}^{\infty} z (u - \hat{u})(s,z) \mathrm{d}z \right|
\end{split}
\end{equation*}
along with \eqref{32a} when $\nu>0$, it follows from \eqref{u_hat_u}, Lemma~\ref{L2} and Lemma~\ref{L3} that
\begin{equation}
\label{E_0}
\int_{y_0}^{\infty} g(y) |E(t,y)| \mathrm{d} y  \leq c(T) \int_{0}^{t} \int_{y_0}^{\infty} g(y) |E(s,y)| \mathrm{d} y \mathrm{d} s +\limsup_{R \to \infty}   V(t,R) \
\end{equation}
for $0 < t < T$. It then remains to show that $\limsup_{R \to \infty}   V(t,R)=0$  for which we consider the cases of Theorem~\ref{T1} and Theorem~\textup{\ref{T2}} separately.

\subsection*{Proof of  Theorem~\ref{T1}}

Let \eqref{18}-\eqref{32} be satisfied and given an initial value  $(v^0,u^0)$ with $v^0 >0$ and $u^0 \in L_1^+(Y, y\mathrm{d}y) \cap L_1\big( Y, (\mu + \beta)(y) G(y) \mathrm{d}y \big)$ consider two corresponding monomer-preserving weak solutions
$(v,u)$ and $(\hat{v}, \hat{u})$ to \eqref{eqv}-\eqref{AW} in the sense of Definition~\ref{D1} such that
\begin{equation}
\label{u_hat_u_1}
u, \hat{u}  \in L_{\infty, \mathrm{loc}}(\mathbb{R}^+, L_1(Y, G(y)\mathrm{d}y)) \cap L_{1, \mathrm{loc}}(\mathbb{R}^+, L_1(Y, (\mu + \beta)(y)G(y)\mathrm{d}y)) \, .
\end{equation} 
One then shows exactly as in the proof of \cite[Theorem 3.1]{LW07} (see equation (39) therein) that \eqref{31}, \eqref{32} along with \eqref{u_hat_u_1} imply
\begin{equation}\label{V}
\lim_{R \to \infty} V(t,R) = 0 \, , \quad  0<t<T \, .
\end{equation}
Applying Gronwall's lemma to  inequality \eqref{E_0} yields  $E \equiv 0$ and Theorem~\ref{T1} follows.

\subsection*{Proof of Theorem~\textup{\ref{T2}}}

To prove Theorem~\textup{\ref{T2}} suppose \eqref{18}-\eqref{p_eta_G} with $g \equiv 1$ (then \eqref{32a} trivially holds) and \eqref{eta_mu}-\eqref{RB}.
Given an initial value  $(v^0,u^0)$ with $v^0 >0$ and  $u^0 \in L_1^+(y, y \mathrm{d}y)$ consider two corresponding monomer-preserving weak solutions
$(v,u)$ and $(\hat{v}, \hat{u})$ to \eqref{eqv}-\eqref{AW} in the sense of Definition~\ref{D1}.
Note that now $G(y) = y-y_0$, hence \eqref{u_hat_u} holds. By Definition~\ref{D1} (see in particular \eqref{monomererhaltend}) we have
\begin{equation} 
\label{y_mu}
\big[ (s,y) \mapsto y\big(1 + \mu(y) \big) u(s,y) \big] \in L_1 \big( (0,t) \times Y \big) \ , \quad t >  0 \, .
\end{equation}
One then shows   as in \cite[Lemma 2.1]{LW07} (using \eqref{monomererhaltend}, \eqref{26}) that 
\begin{equation}
\label{R_beta}
\lim_{R \to \infty} R \int_{0}^{t} \int_{R}^{\infty} \beta(y) u(s,y) \mathrm{d} y \mathrm{d} s = 0 \ , \quad t > 0 \,,
\end{equation}
 the only difference being that when testing \eqref{equ} by $\varphi(y) := R \wedge y$ with $R>y_0$, an additional term
\begin{equation*}
\begin{split}
  \int_{0}^{t} & \int_{y_0}^{\infty} \int_{y_0}^{\infty} \left(  R \wedge (y+z) - R \wedge y - R \wedge z  \right) \, \eta(y,z) \, u(s,y) \, u(s,z) \, \mathrm{d} y \, \mathrm{d} z \, \mathrm{d} s  
\end{split}
\end{equation*}
comes in. However, this term tends to zero as $R\to \infty$ by Lebesgue's theorem since
\eqref{eta_mu} together with \eqref{u_hat_u} and \eqref{y_mu} entail that
\begin{equation*}
\begin{split}
  \int_{0}^{t} & \int_{y_0}^{\infty} \int_{y_0}^{\infty} \left\vert  R \wedge (y+z) - R \wedge y - R \wedge z  \right\vert \, \eta(y,z) \, u(s,y) \, u(s,z) \, \mathrm{d} y \, \mathrm{d} z \, \mathrm{d} s\\
		&\le C_1 \int_{0}^{t} \int_{y_0}^{\infty} \int_{y_0}^{\infty} (y+z) \, \big(\mu(y) + \mu(z) \big) \, u(s,y) \, u(s,z) \, \mathrm{d} y \, \mathrm{d} z \, \mathrm{d} s  
\, <  \, \infty \,.
\end{split}
\end{equation*}
 Finally, exactly as in the proof of \cite[Theorem 3.2]{LW07}, \eqref{y_mu} and \eqref{R_beta} imply \eqref{V}. We may then apply again Gronwall's inequality to \eqref{E_0} (with $g\equiv 1$) and conclude Theorem~\ref{T2}.

\subsection*{Proof of Theorem~\ref{ThE} and Theorem~\ref{ThE2}}

Now, Theorem~\ref{ThE} and Theorem~\ref{ThE2} are consequences of
Theorem~\ref{T1} respectively Theorem~\ref{T2}. Indeed, it was shown in the proof of \cite[Theorem 1.2]{LW07} that \eqref{AK1}, \eqref{AK2} imply \eqref{18}-\eqref{25} and \eqref{31}-\eqref{32} when taking $g(y):= y^{\alpha-1}$ (to satisfy \eqref{32a} one may take $g(y):= y^{(\alpha \vee 1)-1}$) and that \eqref{26}-\eqref{RB} hold under the assumptions of Theorem~\ref{ThE2} while \eqref{eta_mu} follows from \eqref{A} with $\alpha=m$ in this case.
In \cite[Lemma~3.4.]{FL} it was shown  that \eqref{A}, \eqref{B} imply \eqref{eta_W}-\eqref{p_eta_G} for these $g$. This proves Theorem~\ref{ThE} and Theorem~\ref{ThE2}.

\section*{Acknowledgments} We thank the referee for carefully reading the paper and providing helpful comments.



\begin{thebibliography}{99.}%


\bibitem{CLDDMP} V. Calvez, N. Lenuzza, M. Doumic, J.-P. Deslys, F. Mouthon, B. Perthame.
\emph{Prion dynamics with size dependency-strain phenomena.}
J. Biol. Dyn. \textbf{4} (2010), no. 1, 28--42.

\bibitem{CLODLMP} V. Calvez, N. Lenuzza, D. Oelz, J.-P. Deslys, P. Laurent, F. Mouthon, B. Perthame.
\emph{Size distribution dependence of prion aggregates infectivity.}
Math. Biosci. \textbf{217} (2009), no. 1, 89--99.

\bibitem{Cepeda} E. Cepeda.
\emph{Well-posedness for a coagulation multiple-fragmenation equation.}
Differential Integral Equations \textbf{27} (2014),105--136.

\bibitem{DG10} M. Doumic, P. Gabriel. 
\emph{Eigenelements of a general aggregation-fragmentation model.}
 Math. Models Methods Appl. Sci. \textbf{20} (2010), 757--783.

\bibitem{FL} N. Fournier, Ph. Lauren\c cot.
\emph{Well-posedness of Smoluchowski’s coagulation
equation for a class of homogeneous kernels.}
J. Funct. Analysis \textbf{233} (2006), 351--379.


\bibitem{EPW06}  H. Engler, J. Pr\"uss, G. F. Webb.
\emph{Analysis of a model for the dynamics of prions II.}
J. Math. Anal. Appl. \textbf{324} (2006), 98--117.


\bibitem{G15} P. Gabriel.
\emph{Global stability for the prion equation with general incidence.}
Math. Biosci. Eng. \textbf{12} (2015), no. 4, 789--801.

\bibitem{GvDWW07} M. L. Greer, P. van den Driessche, L. Wang, G. F. Webb.
\emph{Effects of general incidence and polymer joining on nucleated polymerzation in a model of prion proliferation.}
SIAM J. Appl. Math. \textbf{68} (2007), no. 1, 154--170.

\bibitem{GPW06}  M. L. Greer, L. Pujo-Menjouet, G. F. Webb.
\emph{A mathematical analysis of the dynamics of prion proliferation.}
J. Theoret. Biol. \textbf{242} (2006), 598--606.



\bibitem{LW07}  Ph. Lauren\c cot, Ch. Walker.
\emph{Well-posedness for a model of prion proliferation dynamics.}
J. Evol. Equ. \textbf{7} (2007), 241--264.

\bibitem{LeisWalker} E. Leis, Ch. Walker.
\emph{Existence of global classical and weak solutions to a prion equation with polymer joining.}
J. Evol. Equ., to appear.

\bibitem{PP-MWZ06}  J. Pr\"uss, L. Pujo-Menjouet, G. F. Webb, R. Zacher.
\emph{Analysis of a model for the dynamics of prions.}
Discrete Contin. Din. Syst. Ser. B \textbf{6} (2006), 225--235.


\bibitem{SW06} G. Simonett, Ch. Walker.
\emph{On the solvability of a mathematical model for prion proliferation.}
J. Math.Anal. Appl. \textbf{324} (2006), 580--603.

\bibitem{W} Ch. Walker.
\emph{Prion proliferation with unbounded polymerization rates.} Pro\-cee\-dings of the Sixth Mississippi State-UBA Conference on Differential Equations and Computational Simulations, 387--397,
Electron. J. Differ. Equ. Conf., 15, Southwest Texas State Univ., San Marcos, TX, 2007. 




\end{thebibliography}
\end{document}